\documentclass[12pt, twoside]{article}
\usepackage[cp1251]{inputenc}
\usepackage{ifthen}
\usepackage{amsmath}
\usepackage{amsfonts}
\usepackage{latexsym}
\usepackage{amsthm}
\usepackage{amssymb}
\usepackage[english]{babel}

\title{Growth description  of $p$th means of the Green potential in the unit ball}

\author{I. Chyzhykov,  M. Voitovych}

\newtheorem{thm}{Theorem}
\newtheorem{cor}[thm]{Corollary}
\newtheorem{prop}[thm]{Proposition}
\newtheorem{atheorem}{Theorem}
\newtheorem{ex}[thm]{Example}

\newtheorem{lem}{Lemma}
\newtheorem{alem}{Lemma}
\newtheorem{rem}{Remark}

\begin{document}
\maketitle

\begin{abstract}
  We describe the growth of $p$th means, $1<p<\frac{2n-1}{2(n-1)}$, of the invariant  Green potential in the unit ball in $\mathbb{C}^n$ in terms of smoothness properties of a measure.
  In particular, a criterion of boundedness  of $p$th means of the potential is obtained, a result of M. Stoll is generalized.
  
  \medskip 
\noindent Keywords: Green potential, unit ball, invariant Laplacian,  $M$-subharmonic function,  Riesz measure

\noindent 2010 Subject classification:  31B25,  31C05.
  \end{abstract}

Firstly, we introduce some definitions and basic notation (\cite{stoll}).
For $n\in \Bbb N$, let $\mathbb{C}^n$ denote the $n$-dimensional complex space with the inner product
$$\langle z,w\rangle=\sum_{j=1}^nz_j\overline{w}_j,~z,w\in\mathbb{C}^n.$$ Let $B=\{z\in\mathbb{C}^n:|z|<1\}$ be the unit ball
and $S=\{z\in\mathbb{C}^n:|z|=1\}$ be the unit sphere, where $|z|=\sqrt{\langle z,z\rangle}$.

For $z,w\in B$, define the \emph{involutive automorphism} $\varphi_w$ of the unit ball $B$ given by
$$\varphi_w(z)=\frac{w-P_wz-(1-|w|^2)^{1/2}Q_wz}{1-\langle z,w\rangle}$$
where $P_0z=0,~P_wz=\frac{\langle z,w\rangle}{|w|^2}w,~w\neq 0$, is the orthogonal projection of $\mathbb{C}^n$ onto the subspace generated by $w$ and $Q_w=I-P_w$.

The invariant Laplacian $\tilde \Delta$ on $B$ is defined by
$$\tilde \Delta f(a)=\Delta(f\circ \varphi_a)(0), $$
where  $f\in C^2(B)$, $\Delta$ is the ordinary Laplacian. It is known that $\tilde \Delta$ is invariant w.r.t. any holomorphic automorphism of $B$ (\cite[Chap.4]{rud}, \cite{stoll}).

The Green's  function for the invariant Laplacian is defined by $G(z,w)=g(\varphi_w(z))$, where $g(z)=\frac{n+1}{2n}\int_{|z|}^1(1-t^2)^{n-1}t^{-2n+1}dt$ (\cite[Chap.6.2]{stoll}).

If $\mu$ is a nonnegative  Borel measure on $B$, the function $G_\mu$ defined by
$$G_\mu(z)=\int_BG(z,w)d\mu(w)$$ is called the \emph{(invariant) Green potential} of $\mu$, provided $G_\mu\not\equiv+\infty.$
It is known that (\cite[Chap.6.4]{stoll}) the last condition is equivalent to
\begin{equation}\label{e:conver}
  \int_B(1-|w|^2)^nd\mu(w)<\infty .
\end{equation}

The concept of the invariant Laplacian naturally implies the following definition (\cite{stoll}, \cite{ulrich}). A function $u$ on B is called \emph{$\mathcal{M} $-sub\-har\-mo\-nic} if it is upper semicontinuos and $\tilde \Delta u\ge 0$ in the sense of distributions. In particular, $-G_\mu$ is $\mathcal{M} $-subharmonic. A function $u$ on B is called \emph{$\mathcal{M} $-harmonic} if $u\in C^2(B)$ and $\tilde \Delta u= 0$.
Unlike the class of plurisubharmonic functions a counterpart of the Riesz decomposition theorem holds  for the class of $\mathcal{M} $-subharmonic functions (see \cite{ulrich}, \cite{stoll}). Due to this theorem, an $\mathcal{M}$-subharmonic function $u$ with the norm $\|u(r\xi)\|_{L^1(S)}$  uniformly bounded on $[0,1)$ can be represented as the difference  of an $\mathcal{M}$-harmonic function and the Green potential of a nonnegative measure satisfying \eqref{e:conver}.   Thus investigations of the Green potentials are very important  in studying the whole class of $\mathcal{M}$-subharmonic functions. Note that  in the case $n=1$ the classes of  $\mathcal{M}$-subharmonic functions and subharmonic functions coincide.

Let $0<p<\infty$, $u$ be a measurable function locally integrable on $B$. We define
$$m_p(r,u)=\left(\int_S\left|u(r\xi)\right|^pd\sigma(\xi)\right)^{\frac{1}{p}}$$
where $d\sigma$ is the Lebesgue measure on $S$ normalized so that $\sigma(S)=1$.

The aim of the paper is to \emph{describe the growth (decrease) of $m_p(r, G_\mu)$ in terms of properties of the measure $\mu$}.
 In the case $n=1$, $p=2$,   this is closely connected to the question   of A. Zygmund, who asked on criterion of boundedness of  $m_2(r,\log|B|)$,  where $B$ is a Blaschke product.  G. MacLane and L. Rubel in \cite{MR} answered his question. 
 Corollary \ref{c:bounded} below is a boundedness criterion for $m_p(r,u)$, $1<p<\frac{2n-1}{2(n-1)}$.
  
  In the case $n>1$, sharp estimates of the growth rate of $m_p(r, G_\mu)$ for the whole class of Borel measures satisfying \eqref{e:conver} are proved by M.~Stoll in \cite{stoll2}.

\begin{atheorem}[\cite{stoll2}] \label{t:a}
   Let $G_\mu$ be the Green potential on $B$.

(1)  If $1\leq p<\frac{2n-1}{2(n-1)}$, then
\begin{equation}\label{e:small_p}
  \lim_{r\to 1-}(1-r^2)^{n(1-1/p)}m_p(r,G_\mu)=0.
\end{equation}

(2)  If $n\geq 2$ and $\frac{2n-1}{2(n-1)}\leq p<\frac{2n-1}{2n-3}$, then
\begin{equation}\label{e:big_p}
\liminf_{r\to 1-}(1-r^2)^{n(1-1/p)}m_p(r,G_\mu)=0.
\end{equation}
\end{atheorem}

The examples in \cite{stoll2} show that the estimates are the best possible in some sense.
We note that similar results for the Green potentials in the unit ball in $\mathbb{R}^n$ are established earlier in \cite{Gard88}  (cf. \cite{Sto83}). For the recent development in the real case we address to the book \cite{stoll16}. The case $n=1$ is studied much more deeper, see e.g. \cite{Li_mp}, \cite{Li_mp1}.

\begin{rem}
 Under additional restrictions on the measure $\mu$ estimates of growth rate $m_p(r, G_\mu)$ can be improved.
\end{rem}

\begin{atheorem}[{\cite[Theorem 3.2]{stoll92}}] \label{t:b}
 Let $\mu$ be a Borel measure on $B$ satisfying
\begin{equation*}
 \int_B (1-|w|^2)^\beta d\mu(w)<\infty
\end{equation*}
 for some real $
\beta \le n$.
\end{atheorem}
(1)  If $1\leq p<\frac{2n-1}{2(n-1)}$ and $-n(1-1/p)<\beta\le n$, then
\begin{equation*}
  \lim_{r\to 1-}(1-r^2)^{\beta -n/p}m_p(r,G_\mu)=0.
\end{equation*}

(2)  If $n\geq 2$ and $\frac{2n-1}{2(n-1)}\leq p<\frac{2n-1}{2n-3}$ and $-n(1-1/p)<\beta\le n$, then
\begin{equation*}
\liminf_{r\to 1-}(1-r^2)^{\beta-n/p}m_p(r,G_\mu)=0.
\end{equation*}

\begin{rem}
 Theorem \ref{t:b} implies Theorem \ref{t:a} for $\beta=n$.
\end{rem}

\begin{rem} \label{r:p=1} It follows from  results    of \cite{ulrich} (see also \cite{stoll}) that $$m_1(r, G_\mu)=o(1) \quad r\to 1-.$$ So we omit the case $p=1$.
\end{rem}

\begin{rem} \label{r:p_large} It is shown in  \cite[Example 2]{stoll2}  that for each $n\ge 2$ there exists a discrete measure $\mu$ satisfying \eqref{e:conver} such that $$\limsup_{r\to 1-}(1-r^2)^{n(1-1/p)}m_p(r,G_\mu)=\infty$$  for all $p\ge \frac
{2n-1}{2(n-1)}$.
\end{rem}

In view of Remarks \ref{r:p=1} and \ref{r:p_large}, we confine to the case $1<p< \frac
{2n-1}{2(n-1)}$.

Theorem \ref{t:a} gives the maximal growth rate of the $p$th mean of the Green potentials, but does not take into account particular properties of a measure $\mu$. It is appeared that smoothness properties of so called {\it complete measure} (in the sense of Grishin \cite{Gr1}, \cite{Chy08ms}, \cite{chyzh}) or the {\it related measure} (see \cite{Gard84}) of a subharmonic function allow to describe its growth.
Here we just note that in the case when  $n=1$ and $u=-G_\mu$, the complete measure $\lambda=\lambda_u$ of $u$ is the weighted Riesz measure $d\lambda (z)=(1-|z|)d\mu(z)$. In particular, results from \cite{chyzh} imply


\begin{atheorem}[\cite{chyzh}] Let $\gamma\in(0,1],~p\in(1,\infty)$, $n=1$,  and $\mu$ be a Borel measure satisfying \eqref{e:conver}. Let $\lambda$ be defined as above. Necessary and sufficient that
$$m_p(r,G_\mu)=O\left((1-r)^{\gamma-1}\right),~r\to  1-,$$

hold is that
$$\int\limits_0^{2\pi}\lambda^p(\{\rho e^{i\theta}\in{{B}}:\rho \geq 1-\delta,~|\theta-\varphi|\leq\pi\delta\})d\varphi=O(\delta^{p\gamma}),\quad 0<\delta<1.$$
\end{atheorem}

Define for $a,b\in\bar{B}$ the \emph{anisotropic metric} on $S$ by $d(a,b)=|1-\langle a,b\rangle|^{1/2}$ (\cite[Chap.5.1]{rud}).

For $\xi\in S$ and $\delta >0$ we denote  $$C(\xi,\delta)=\{z\in B:d(z,\xi)<\delta^{1/2}\}, \quad  D(\xi,\delta)=\{z\in B:d(z,\xi)<\delta\}, $$ and $d\lambda(z)=(1-|z|)^nd\mu(z).$

The following theorem is our main result. As we can see later, it generalizes Theorem \ref{t:a}(1) and Theorem \ref{t:b}(1).
\begin{thm} \label{t:main}
   Let $n>1$,  $1<p<\frac{2n-1}{2(n-1)}$, $0\leq\gamma<2n$, $\mu$ be a  Borel measure satisfying  (\ref{e:conver}). Then
   \begin{equation}\label{e:growth}
m_p\left(r,G_\mu\right)=O\left((1-r)^{\gamma-n}\right),~r\uparrow 1
\end{equation}
holds if and only if
\begin{equation}\label{e:smooth}
\left(\int_S\lambda^p\left(C(\xi,\delta)\right)d\sigma(\xi)\right)^{\frac{1}{p}}=O\left(\delta^\gamma\right),~0<\delta<1.
\end{equation}
\end{thm}

As a corollary we obtain a criterion of boundedness of the invariant Green potential.
\begin{cor} \label{c:bounded}
   Let $n>1$, $1<p<\frac{2n-1}{2(n-1)}$, $\mu$ be a  Borel measure satisfying  (\ref{e:conver}). Then
   \begin{equation*}\label{e:growth_b}
m_p\left(r,G_\mu\right)=O\left(1\right), \quad 0<r< 1
\end{equation*}
 if and only if
\begin{equation}\label{e:smooth_cor}
\left(\int_S\lambda^p\left(C(\xi,\delta)\right)d\sigma(\xi)\right)^{\frac{1}{p}}=O\left(\delta^n\right),~0<\delta<1.
\end{equation}
\end{cor}


\begin{rem}
 For $\gamma\in(n,2n)$, Theorem \ref{t:main} gives necessary and sufficient conditions for decrease of the Green potential.
\end{rem}

\begin{ex}
 If $\mu$ is the Lebesgue measure on $B$, then
  \begin{equation}\label{e:lebeg_smoth}
   \lambda^p\left(C(\xi,\delta)\right)=O(\delta^{n+1})
\end{equation}
i.e. the assumption \eqref{e:smooth} holds with $\gamma=n+1$, thus $m_p\left(r,G_\mu\right)=O(1-r)$ as $r\to1-$ and $n>1$.
The latter relation is valid in the case $n=1$ as well, which can be checked directly.

    The estimate \eqref{e:lebeg_smoth} follows from the next remarks. Firstly, the radial projection of $C(\xi,\delta)$ on $S$ has $(2n-1)$-dimensional measure $\sigma_\delta=c\delta^n$ (\cite[Prop. 5.1.4]{rud}). Secondly,  by the definition, $C(\xi,\delta)\subset \{z\in B: |z|\ge 1-\delta\}$.
\end{ex}

It is sometimes suitable to have an ``o''-analog of Theorem \ref{t:main}.

\begin{thm} \label{t:smallo}
   Let $n>1$, $1<p<\frac{2n-1}{2(n-1)}$, $0\leq\gamma< 2 n$, $\mu$ be a  Borel measure satisfying  (\ref{e:conver}). Then
   \begin{equation}\label{e:ogrowth}
m_p\left(r,G_\mu\right)=o\left((1-r)^{\gamma-n}\right),~r\to  1-
\end{equation}
holds if and only if
\begin{equation}\label{e:osmooth}
\left(\int_S\lambda^p\left(C(\xi,\delta)\right)d\sigma(\xi)\right)^{\frac{1}{p}}=o\left(\delta^\gamma\right),~\delta\to0+.
\end{equation}
\end{thm}

The following elementary proposition is useful.
\begin{prop} \label{p:smooth}
  Let $n\in \mathbb{N}$, $\nu$ be a finite  Borel measure on $B$. Then
\begin{equation}\label{e:smoothness}
\left(\int_S\nu^p(C(\xi,\delta))d\sigma(\xi)\right)^{\frac{1}{p}}=o(\delta^{\frac{n}{p}}), \quad \delta\to 0+.
\end{equation}
\end{prop}

\begin{rem}
 Theorem \ref{t:smallo} and Proposition \ref{p:smooth} imply Theorem A(1) as a corollary for $p>1$.
In Section  \ref{s:proofs} we show that Theorem \ref{t:smallo} and Proposition \ref{p:smooth} imply Theorem B(1) as well.
\end{rem}

In the sequel, the symbol $c$ stands for positive constants which depend on the parameters  indicated in the parentheses,
$a\asymp b$ means that there are positive constants $c'$ and $c''$ such that $c' a< b< c'' a$ holds.


\section{Auxiliary results}

The following lemma gives some basic properties of $g$ which will be needed later.

\begin{alem}[\cite{stoll}] 
 Let $0<\delta<\frac{1}{2}$ be fixed. Then $g$ satisfies the following relations:
$$g(z)\geq\frac{n+1}{4n^2}(1-|z|^2)^n,~z\in B,$$
\begin{equation}\label{e:lem_a2}
g(z)\leq c(\delta)(1-|z|^2)^n,~z\in B,|z|\geq\delta,
\end{equation}
where $c(\delta)$ is a positive constant. Furthermore,  if $n>1$ then
\begin{equation}\label{e:g_asymp}
g(z)\asymp|z|^{-2n+2},~ \quad
|z|\leq\delta.
\end{equation}
\end{alem}

We need the following multidimensional generalization of Lem\-ma~1 from~\cite{chyzh}.
\begin{lem} \label{l:cn}
 Let $\nu$ be a finite positive  Borel measure on $S$,  $0<\delta<\frac 12$, and $p\ge 1$. Then
$$\int_S\nu^{p-1}(D(\xi,\delta) )d\nu(\xi) \leq
\frac{N^p}{\delta^{2n}}\int_S\nu^{p}(D(\xi,\delta))d\sigma(\xi),$$
where $N$ is a positive constant independent of $p$ and $\delta$.
\end{lem}

\begin{proof}[Proof of the lemma] First, we prove the statement for $p=1$.
Since (\cite[Prop. 5.1.4]{rud}) $\sigma(D(\xi, \delta))\asymp \delta^{2n}$, one has
 \begin{gather}\label{e:volume_n}
\int_Sd\nu(\xi)\le \frac{c}{\delta^{2n}} \int_S d\nu(\xi) \int_{D(\xi, \delta)}d\sigma(t).\end{gather}
Let $\Theta\colon \Pi \to S$ be the spherical coordinates on the unit sphere, where $\Pi=[0, \pi]^{2n-2}\times [0, 2\pi)$. 
 Since $\Theta$ is periodic in each variable, we consider $\Theta$  on $\mathbb{R}^{2n-1}$.
We set $\Pi '=[-\frac{\pi}{2}, \frac{3\pi}2]^{2n-2}\times [-\pi, 3\pi)$. Then, using Fubini's theorem and the periodicity of the Jacobian $\det \Theta'$, we deduce
 \begin{gather*}
 \int\limits_S d\nu(\xi) \int\limits_{D(\xi, \delta)}d\sigma(t)=\int\limits_\Pi d\nu(\Theta(x)) \int\limits_{\begin{substack} {{d(\Theta(x),\Theta(y))<\delta} \\ |x-y|<\frac\pi 2 } \end{substack}}  |\det \Theta'(y)| dy\\
 \le\int\limits_{\Pi'} dy \int\limits_{\begin{substack} {{d(\Theta(x),\Theta(y))<\delta} \\ |x-y|<\frac\pi 2 } \end{substack}} |\det \Theta'(y)| d\nu(\Theta(x)) \\ =2^{2n-1} \int\limits_\Pi dy \int\limits_{\begin{substack} {{d(\Theta(x),\Theta(y))<\delta} \\ |x-y|<\frac\pi 2 } \end{substack}} |\det \Theta'(y)| d\nu(\Theta(x)) \\
 =2^{2n-1}  \int\limits_{S} d\sigma(t)\int\limits_{D(t, \delta)}d\nu(\xi)=2^{2n-1}\int\limits_S {\nu(D(t, \delta))}d\sigma(t)
  .\end{gather*}
Substituting this estimate into \eqref{e:volume_n}, we obtain the statement of the lemma in the case $p=1$ with $N=c 2^{2n-1}$.

Let now $p>1$. We define $d\nu_1(\xi)=\nu^{p-1}(D(\xi, \delta))d\nu(\xi)$. Then
applying  the statement of the lemma for $p=1$ we get
\begin{gather} \nonumber
\int_S\nu^{p-1}(D(\xi, \delta))d\nu(\xi)
=\int_Sd\nu_1(\xi)\leq
2^{2n-1}c\int_S\frac{\nu_1(D(t, \delta))}{\delta^{2n}}d\sigma(t)\\ \nonumber
=\frac{2^{2n-1} c}{\delta^{2n}}\int_S\left(\int_{D(t, \delta)}\nu^{p-1}(D(\xi, \delta))d\nu(\xi)\right)d\sigma(t)\\
\leq\frac{2^{2n-1} c}{\delta^{2n}} \int_S{\nu^{p-1}(D(t, 2\delta))} \nonumber
\nu(D(t, \delta))d\sigma(t) \\
\leq \frac{2^{2n-1} c}{\delta^{2n}} \int_S{\nu^{p}(D(t, 2\delta))}d\sigma(t). \label{e:nu_p_est}
\end{gather}
Let $\{ t_1, \dots, t_N\}\subset S$ be a finite $\delta$-net for $D(e_1, 2\delta)$ where $e_1=(1, 0, \dots,0)$ and $N$ depends on $n$ only, i.e.
$\bigcup_{k=1}^N D(t_k, \delta)\supset D(e_1, 2\delta).$ Then $t_k$ can be represented in the form $t_k=\tau_k(e_1)$, where $\tau_k\in U(n)$, $k\in \{1, \dots, N\}$, are unitary transformations of $\mathbb{C}^n$. 
Taking into account that the measure $\sigma$ is invariant w.r.t. the elements of $U(n)$, we deduce
\begin{gather*}
\int\limits_S\nu^p(D(t, 2\delta))d\nu(t)\leq
\int\limits_S{\nu^p\left(\bigcup_{k=1}^N D(\tau_k(t), \delta)\right)}d\sigma(t)\\
\leq N^{p-1}\sum_{k=1}^{N}\int\limits_S {\nu^p\left(D(\tau_k(t), \delta)\right)}d\sigma(t)= N^{p-1}\sum_{k=1}^{N}\int\limits_S {\nu^p\left(D(t, \delta)\right)}d\sigma(t) \\
=N^p\int\limits_S {\nu^p(D(t, \delta))}d\sigma(t).
\end{gather*}
Taking into account \eqref{e:nu_p_est} we finish the proof of the lemma.
\end{proof}

\section{Proofs of the main results} \label{s:proofs}
\begin{proof}[Proof of Theorem \ref{t:main}]
{\it Sufficiency.} 
Denote $$B^*\Bigl(z,\frac{1}{4}\Bigr)=\Bigl\{w\in B:|\varphi_w(z)|<\frac{1}{4}\Bigr\}.$$
Let us estimate the absolute values   of
$$u_1(z):=\int\limits_{B^*\left(z,\frac{1}{4}\right)}G(z,w)d\mu(w)~~\text{and}~~
u_2(z):=\int\limits_{B\setminus B^*\left(z,\frac{1}{4}\right)}G(z,w)d\mu(w).$$
We start with $u_1$. By definition
\begin{gather*}
0\leq u_1(z)=\int_{B^*\left(z,\frac{1}{4}\right)}G(z,w)d\mu(w)
=\int_{B^*\left(z,\frac{1}{4}\right)}g(\varphi_w(z))d\mu(w).
\end{gather*}
By (\ref{e:g_asymp}) we have $g(z)\geq c|z|^{-2n+2}$ for $|z|\leq\frac{1}{4}$ and some positive constant $c$. Thus,
\begin{gather*}
|u_1(z)|\leq c\int_{B^*\left(z,\frac{1}{4}\right)}\left|\varphi_w(z)\right|^{-2n+2}d\mu(w).
\end{gather*}

Denote $z=r\xi$, where $r=|z|$, $\frac{1}{2}<r<1$ and $w=|w|\eta,~\xi,\eta\in S$. Let $$K(z,\sigma_1,\sigma_2)=\left\{w\in B:|r-|w||\leq\sigma_1, d(\xi,\eta)\leq\sigma_2\right\}.$$ Now let us proof the inclusion
\begin{equation}\label{e:inclusion}
B^*\left(z,\frac{1}{4}\right)\subset K\left(z,c_1(1-r),c_2(1-r)^\frac{1}{2}\right)
\end{equation}
where $c_1,c_2$ are positive constants.
Since (\cite[p.11]{stoll})  for $w
\in B^*(z, \frac14)$
\begin{equation}\label{e:mod_inv}
  \left(1-\frac{(1-|w|^2)(1-r^2)}{|1-\langle z,w\rangle|^2}\right)^\frac{1}{2}=|\varphi_w(z)|<\frac{1}{4},
\end{equation}
$$A(z):=\Bigl\{ w\in B: \left(1-\frac{(1-|w|^2)(1-r^2)}{(1- r|w|)^2}\right)^\frac{1}{2}< \frac
 14 \Bigr\} \supset B^*\Bigl(z, \frac 14\Bigr).$$
In order to find $c_1$ it is enough to check that
$$\partial A(z)\subset \overline{K} (z,c_1(1-r), c_2(1-r)^{\frac12}).$$
For $z \in  \partial A(z)$ one has $ \frac{(1-|w|^2)(1-r^2)}{(1-r|w|)^2}=\frac{15}{16}$. Solving this equation, we find
$|w_1|=\frac{4r+1}{r+4},~|w_2|=\frac{4r-1}{4-r}$.
So
\begin{gather*}
|r-|w_1||=|w_1|-r=\frac{1+r}{r+4}(1-r)<\frac{2}{5}(1-r),\\
|r-|w_2||=r-|w_2|=\frac{1+r}{4-r}(1-r)<\frac{2}{3}(1-r).
\end{gather*}
Thus, $|r-|w||< \frac 23 (1-r)$ for $w\in B^*(z, \frac 14)$.
Then, to estimate $c_2$ we deduce
\begin{gather*}
d^2(\xi,\eta)=|1-\langle\xi,\eta\rangle|<\frac{1}{r|w|}|1-r|w|\langle\xi,\eta\rangle|\\
<\frac{1}{r|w|}\left(\frac{16}{15}(1-|w|^2)(1-r^2)\right)^\frac{1}{2}\\
<\frac{1}{r\left(-2/3+5r/3\right)}\left(\frac{16}{15}\left(1-\left(-\frac{2}{3}+
\frac{5}{3}r\right)^2\right)(1-r^2)\right)^{\frac{1}{2}}\\
=\frac{4}{\sqrt{3}r(5r-2)}\left((5r+1)(1+r)\right)^{\frac{1}{2}}(1-r)<32(1-r),~
\end{gather*}
where $\frac{1}{2}<r<1$. So  \eqref{e:inclusion} holds with  $c_1=\frac{2}{3}$ and $c_2=4\sqrt{2}$.
We denote
\begin{gather*}
 K(z):= K\left(z,\frac 23(1-r), 4\sqrt{2} (1-r)^\frac{1}{2}\right),  \\
  \tilde K(z):= K\left(z,\frac 23(1-r), 8\sqrt{2} (1-r)^\frac{1}{2}\right).
\end{gather*}
   H\"{o}lder's inequality and inclusion (\ref{e:inclusion}) imply
\begin{gather*}
I_1:=\int_S|u_1(r\xi)|^pd\sigma(\xi)\\ \leq c_3\int_S\left(\int_{B^*\left(r\xi,\frac{1}{4}\right)}|\varphi_w(r\xi)|^{-2n+2}d\mu(w)\right)^pd\sigma(\xi)\\
\leq c_3 \int_S\int_{B^*\left(r\xi,\frac{1}{4}\right)}|\varphi_w(r\xi)|^{-p(2n-2)}d\mu(w)
\mu^{p-1}\left(B^*\left(r\xi,\frac{1}{4}\right)\right)d\sigma(\xi)\\
\leq c_3 \int\limits_S\int\limits_{K\left(r\xi\right)}\frac{d\mu(w)}{|\varphi_w(r\xi)|^{p(2n-2)} } \mu^{p-1}\left(K\left(r\xi\right)\right)d\sigma(\xi) \\
\le  c_3 \int\limits_S\int\limits_{K\left(r\xi\right)}\frac{\mu^{p-1}\left(\tilde K\left(r\eta\right)\right)}{|\varphi_w(r\xi)|^{p(2n-2)} } d\mu(|w|\eta) d\sigma(\xi)
\end{gather*}
where $c_3=c_3(p)$.
Then, by  Fubini's theorem (cf. the proof of Lemma~\ref{l:cn}) we deduce ($z=r\xi$, $w=|w|\eta$)
\begin{gather} \nonumber
I_1\le 
c_4(n,p) \iint\limits_{\begin{substack} {\eta\in S\\ ||w|-r|< \frac 23(1-r)\\ d(\xi, \eta)<4\sqrt{2}(1-r)^{1/2}} \end{substack}} \frac{\mu^{p-1}\left(\tilde K\left(r\eta\right)\right)}{|\varphi_w(r\xi)|^{p(2n-2)} } d\mu(|w|\eta) d\sigma(\xi)\\ \le c_4(p,n)
\int\limits_{||w|-r|< \frac 23(1-r)} \mu^{p-1}\left(\tilde K\left(r\eta\right)\right)\int\limits _S \frac{d\sigma(\xi)}{|\varphi_w(r\xi)|^{p(2n-2)}}d\mu(w). \label{e:i1_est1}
\end{gather}
Applying subsequently \eqref{e:mod_inv}, \eqref{e:lem_a2} and  Lemma 5  (\cite{stoll2}), we obtain that for $1<p<\frac{2n-1}{2(n-1)}$
\begin{gather*}
  \int _S \frac{d\sigma(\xi)}{|\varphi_w(r\xi)|^{p(2n-2)}} =\int _S \frac{d\sigma(\xi)}{|\varphi_{r\xi}(w)|^{p(2n-2)}}   \\ \le
  \int_Sg^p(\varphi_{r\xi}(w))d\sigma(\xi )\leq\frac{c_5(1-|w|^2)^{np}}{(1-r^2)^{n(p-1)}}, \quad \frac 12 < r<1.
\end{gather*}
Substituting the estimate of the inner integral into \eqref{e:i1_est1} we get
\begin{gather} \nonumber
I_1\leq c_4 \int\limits_{||w|-r|< \frac 23 (1-r)} \frac{c_5(1-|w|^2)^{np}}{(1-r^2)^{n(p-1)}}\mu^{p-1}\left(\tilde K\left(r\eta\right)\right)d\mu(|w|\eta)\\
\leq c_6(1-r)^n\int\limits_{||w|-r|< \frac 23 (1-r)} \mu^{p-1}\left(\tilde K\left(r\eta\right)\right)d\mu(|w|\eta). \label{e:i1_fin_est}
\end{gather}

To obtain the final estimate of $I_1$, for a fixed $r\in (\frac12, 1)$, we define the  measure $\nu_1$ on the balls by
$$\nu_1(D(\eta,t))=\lambda\Bigl(\Bigr\{\rho\zeta\in B:|\rho-r|<\frac 23(1-r),d(\zeta,\eta)<t\Bigr\}\Bigr).$$
It can be expanded to the family   of all Borel sets on $B$ in the standard way.
It is clear that $$\nu_1(D(\eta,t))\asymp (1-r)^n\mu\Bigl(\Bigl\{\rho\zeta\in B:|\rho-r|<\frac 23(1-r),d(\zeta,\eta)<t\Bigr\}\Bigr).$$
By using of \eqref{e:i1_fin_est} and Lemma \ref{l:cn} we get
\begin{gather*}
I_1 \le \frac{c_7}{(1-r)^{n(p-1)}}   \int\limits_{||w|-r|< \frac 23 (1-r)} \lambda^{p-1}\left(\tilde K\left(r\eta\right)\right)d\lambda(|w|\eta)\\ =
\frac{c_7}{(1-r)^{n(p-1)}}  \int_S\nu_1^{p-1}\left(D(\eta, 8\sqrt{2}(1-r)^{\frac{1}{2}})\right)d\nu_1(\eta)\\ \allowdisplaybreaks
\leq\frac{c_7  N^p}{(128)^{n}(1-r)^{np}}\int_S\nu_1^p\left(D(\eta, 8\sqrt{2}(1-r)^{\frac{1}{2}})\right)d\sigma(\eta)\\
=\frac{c_8(n,p)}{(1-r)^{np}}\int_S\lambda^p\left( \tilde K(r\eta)\right)d\sigma(\eta).
\end{gather*}

Note that if $\rho\zeta\in \tilde{K}(r\eta)$  then
\begin{equation}\label{e:anis_inclus}
|1-\langle \rho\zeta ,\eta\rangle|\leq\left|1-\left\langle {\zeta},\eta\right\rangle\right|+(1-\rho)\left|\left\langle \zeta,\eta\right\rangle\right|\leq (4c_2^2+c_1+1)(1-r).
  \end{equation}
Hence, by the assumption of the theorem
\begin{gather} \nonumber
I_1\leq c_8(1-r)^{-np}\int_S\lambda^p\left(C(\eta,(4c_2^2+c_1+1)(1-r))\right)d\sigma(\eta)\\ \leq
c_{9}(1-r)^{p(\gamma-n)}. \label{e:i1est}
\end{gather}
Let us estimate $$u_2(z)=\int_{B}G(z,w)(1-|w|)^{-n}d\tilde{\lambda}(w)$$ where
$d\tilde{\lambda}(w)=(1-|w|)^n\chi_{B\setminus B^*\left(z,\frac{1}{4}\right)}(w)d\mu(w)$, $\chi_E$ is the characteristic function of a set $E$.
We may assume that $|z|\ge \frac12$.

We denote $$E_k=E_k(z)=\left\{w\in B:\left|1-\left\langle\frac{z}{|z|},w\right\rangle\right|<2^{k+1}(1-|z|)\right\}, \quad
k\in\mathbb{Z}_+.$$
Then for $w\in E_{k+1}(z)\setminus E_k(z),~\frac{1}{2}\leq|z|<1$
$$|1-\langle z,w\rangle|\geq|z|\left|1-\left\langle\frac{z}{|z|},w\right\rangle\right|-(1-|z|)\geq\left(|z|2^{k+1}-1\right)(1-|z|).$$
Combining  Lemma A with the equality in   \eqref{e:mod_inv} for $z\in B$ such that $|z|\ge\frac{1}{2}$ we get that
$0\leq G(z,w)\leq c_{10}\left(\frac{(1-|w|^2)(1-|z|^2)}{|1-\langle z,w\rangle|^2}\right)^n$ holds. So
\begin{gather*}
|u_2(z)|\leq c_{10}\int_B\left(\frac{(1+|w|)(1-|z|^2)}{|1-\langle z,w\rangle|^2}\right)^nd\tilde{\lambda}(w)\\
\leq \sum_{k=1}^{[\log _2 \frac 1{1-r}]} c_{10}\int_{E_{k+1}\setminus E_k}\left(\frac{(1+|w|)(1-|z|^2)}{(|z|2^{k+1}-1)^2(1-|z|)^2}\right)^nd\tilde{\lambda}(w) \\ +
c_{10}\int_{E_1}\left(\frac{(1+|w|)(1-|z|^2)}{(1-|z|)^2}\right)^nd\tilde{\lambda}(w)\\
\leq\sum_{k=1}^\infty \int_{E_{k+1}\setminus E_k}\frac{4^nc_{10}}{\left(2^{2(k-1)}(1-|z|)\right)^n}d\tilde{\lambda}(w)+
\int_{E_1}\frac{4^nc_{10}}{(1-|z|)^n}d\tilde{\lambda}(w)\\
\leq \frac{4^nc_{10}}{(1-|z|)^n}\left(\sum_{k=1}^\infty\frac{\tilde{\lambda}\left(E_{k+1}\right)}{2^{2n(k-1)}}+
\tilde{\lambda}\left(E_1\right)\right)
\leq\frac{4^nc_{10}}{(1-|z|)^n}\sum_{k=1}^\infty\frac{\tilde{\lambda}\left(E_{k}\right)}{2^{2n(k-2)}}.
\end{gather*}
Fix any $\alpha\in (\frac{\gamma}{n}, 2)$. By  H\"{o}lder's inequality ($\frac{1}{p}+\frac{1}{q}=1$)
\begin{gather} \nonumber
|u_2(z)|^p\leq\frac{4^{np}c^p_{10}}{(1-|z|)^{np}}\sum_{k=1}^\infty
\frac{\tilde{\lambda}^p\left(E_{k}\right)}{2^{\alpha np(k-2)}}
\left(\sum_{k=1}^\infty\frac{1}{2^{{(2-\alpha)nq(k-2)}}}\right)^{p/q}\\ 
=\frac{4^{np}c^p_{10}}{(1-|z|)^{np}}\frac{2^{4np}}{\left(2^{{(2-\alpha)nq}}-1\right)^{p/q}}
\sum_{k=1}^\infty\frac{\tilde{\lambda}^p\left(E_k\right)}{2^{{\alpha npk}}}
=\frac{c_{11}(n,p, \alpha)}{(1-|z|)^{np}}
\sum_{k=1}^\infty\frac{\tilde{\lambda}^p\left(E_k\right)}{2^{{\alpha npk}}}. \label{e:u2_holder}
\end{gather}
Therefore
\begin{gather*}
\int_S|u_2(r\xi)|^pd\sigma(\xi)\leq\frac{c_{11}}{(1-r)^{np}}
\sum_{k=1}^\infty\int_S\frac{\tilde{\lambda}^p\left(E_k(r\xi)\right)}{2^{{\alpha npk}}}d\sigma(\xi)\\
=\frac{c_{11}}{(1-r)^{np}}
\sum_{k=1}^\infty\frac{1}{2^{\alpha{npk}}}
\int_S\tilde{\lambda}^p\left(C\left(\frac{z}{|z|},2^{k+1}(1-r)\right)\right)d\sigma(\xi)\\
\leq\frac{c_{12}}{(1-r)^{np}}\sum_{k=1}^\infty\frac{2^{p\gamma(k+1)}(1-r)^{\gamma p}}{2^{\alpha npk}}\\
=\frac{c_{12}}{(1-r)^{p(n-\gamma)}}\frac{2^{p\gamma}}{2^{p(\alpha n-\gamma)}-1}=\frac{c_{13}(n,p,\gamma)}{(1-r)^{p(n-\gamma)}}.
\end{gather*}
The latter inequality together with \eqref{e:i1est} completes the proof of the sufficiency.

{\it Necessity.}
By Lemma A
\begin{gather*}
G_\mu(r\xi)=\int_Bg(\varphi_w(r\xi))d\mu(w)\geq
\int_B\frac{n+1}{4n^2}\frac{(1-|w|^2)^n(1-r^2)^n}{|1-\langle r\xi,w\rangle|^{2n}}d\mu(w)\\
\geq\int_{C(\xi,1-r)}\frac{n+1}{4n^2}\frac{(1-|w|^2)^n(1-r^2)^n}{|1-\langle r\xi,w\rangle|^{2n}}d\mu(w)\\
=\int_{C(\xi,1-r)}\frac{n+1}{4n^2}\frac{(1+|w|)^n(1-r^2)^n}{|1-\langle r\xi,w\rangle|^{2n}}d\lambda(w).
\end{gather*}

Since for $w\in C(\xi,1-r)$
$$|1-\langle z,w\rangle|\leq\left|1-\left\langle \frac{z}{|z|},w\right\rangle\right|+\left|\left\langle \frac{z}{|z|}-z,w\right\rangle\right|\leq 2(1-|z|),$$
we have
\begin{gather*}
|G_\mu(r\xi)|\geq\frac{n+1}{4^{n+1}n^2}\frac {\lambda(C(\xi,1-r))}{(1-r)^{n}}.
\end{gather*}

From the assumption of the theorem it follows that
\begin{gather*}
\left(\frac{n+1}{2^{2(n+1)}n^2}\right)^p\int_{S}\frac{\lambda^p(C(\xi,1-r))}{(1-r)^{np}}d\sigma(\xi)\\
\leq\int_{S}|G_\mu(r\xi)|^pd\sigma(\xi)\leq c_{13}^p(1-r)^{p(\gamma-n)}.
\end{gather*}

Thus
$$\int_S\lambda^p(C(\xi,1-r))d\sigma(\xi)\leq c_{13}^p(1-r)^{p\gamma}, \quad 0<r<1.$$
\end{proof}

\begin{proof}[Proof of Theorem \ref{t:smallo}]
  The proof of the {\it necessity} literally repeats that of Theorem~\ref{t:main}.

  In the {\it sufficiency}  part the estimate of $u_1$ is quite similar. In order to estimate $u_2$ we note that, by the definition of $E_k(z)$,
  $$ 1-|w|\leq\left|1-\left\langle \frac{z}{|z|},w\right\rangle\right| \le 2^{k+1} (1-|z|) \le 2\sqrt{1-|z|},$$
   for $w
   \in E_k(z)$, $1\le k\le \frac 12 \log_2 \frac{1}{1-|z|}.$ Thus \eqref{e:osmooth} implies as $r\to1-$
      \begin{equation}\label{e:small_k}
\sum_{k=1}^{\bigl[\frac 12 \log_2 \frac{1}{1-r}\bigr]}  \int_S \tilde{\lambda}^p\left(E_k(r\xi)\right) d\sigma(\xi)= o\biggl(
   \sum_{k=1}^{\bigl[\frac 12 \log_2 \frac{1}{1-r}\bigr]} 2^{p\gamma(k+1)} (1-r)^{\gamma p} \biggr) .
      \end{equation}
   Applying estimates \eqref{e:u2_holder} and \eqref{e:small_k} we deduce
   \begin{gather*}
\int_S|u_2(r\xi)|^pd\sigma(\xi) \\ \leq \frac{c_{12}(n,p)}{(1-r)^{np}}
\biggl( \sum_{k=1}^{\bigl[\frac 12 \log_2 \frac{1}{1-r}\bigr]}+ \sum_{k=\bigl[\frac 12 \log_2 \frac{1}{1-r}\bigr]+1}^{\infty} \biggr) \int_S\frac{\tilde{\lambda}^p\left(E_k(r\xi)\right)}{2^{{\alpha npk}}}d\sigma(\xi)\\
=\frac{o(1)}{(1-r)^{np-\gamma p}}
\sum_{k=1}^{\bigl[\frac 12 \log_2 \frac{1}{1-r}\bigr]}\frac{2^{p\gamma(k+1)}}{2^{{\alpha npk}}} \\ +
\frac{c_{13}}{(1-r)^{np-\gamma p}} \sum_{k=\bigl[\frac 12 \log_2 \frac{1}{1-r}\bigr]+1}^{\infty} \frac{2^{p\gamma(k+1)}}{2^{\alpha npk}}\\
=\frac{o(1)}{(1-r)^{p(n-\gamma)}} + \frac{c_{14}}{(1-r)^{p(n-\gamma)}} \frac{1}{(1-r)^{\frac p2 (\gamma- \alpha n)}}\\ =\frac{o(1)}{(1-r)^{p(n-\gamma)}}, \quad r\uparrow 1.
\end{gather*}
   Since the integral  of  $|u_1|^p$ admits the same estimate, it completes the proof of Theorem \ref{t:smallo}.
  \end{proof}

  \begin{proof}[Proof of  Proposition \ref{p:smooth}] 
Let   $\nu(B)=M>0$. Applying Fubini's theorem (cf. the proof of Lemma \ref{l:cn}) we deduce
\begin{gather*}
\int_S\nu^p(C(\xi, \delta))d\sigma(\xi)\leq
M^{p-1}\int_S\nu(C(\xi, \delta))d\sigma(\xi)\\
=M^{p-1}\int_S\int_{C(\xi, \delta)}d\nu(z)\, d\sigma(\xi)\\ \le M^{p-1} 2^{2n-1} \int_{1-\delta\le |z|<1} d\nu(z)\int_{C(z, \delta)}  d\sigma(\xi) \\
\le c M^{p-1}  2^{2n-1} \delta^n \int_{1-\delta\le |z|<1} d\nu(z)= o(\delta^n), \quad \delta\downarrow 0. \end{gather*}

\end{proof}
\begin{proof}[An alternative proof of Theorem \ref{t:b}(1)]
 Suppose that the assump\-ti\-ons of Theorem \ref{t:b} holds.
We set $d\nu(w)=(1-|w|)^\beta \, d\mu(w)$.
We deduce \begin{gather*}
          \lambda(C(\xi, \delta)) =\int_{C(\xi, \delta)} (1-|w|)^n d\mu(w)\\
\le C \delta^{n-\beta} \int_{C(\xi, \delta)} (1-|w|)^{\beta} d\mu(w)\le C \delta^{n-\beta} \nu(C(\xi, \delta)).
         \end{gather*}
 By Proposition \ref{p:smooth}, $\int_S\nu^p(C(\xi, \delta))d\sigma(\xi)=o(\delta^n),$ $\delta \to 0+$.
Hence,
$$\biggl(\int_S \lambda^p(C(\xi, \delta))\, d\sigma(\xi)\biggr)^{\frac1p}=o(\delta^{n-\beta+\frac np}), \quad
\delta \to 0+.$$
Since $n-\beta+\frac np<2n$, by Theorem \ref{t:smallo} we get the  assertion of Theorem~\ref{t:b}(1).
\end{proof}

\section{Further results}
Let  $\Phi\colon [0,1]\to [0,+\infty)$ be an increasing function such that  for some $\gamma\in (0,2n)$ and all $t, \delta$ with  $0<\delta<1$, $0<t\delta< 1$ we have  $\Phi(t\delta)\le t^{\gamma} \Phi(\delta)$.
Applying the similar arguments one can show that
\begin{equation*}
m_p\left(r,G_\mu\right)=O\left(\frac{\Phi (1-r)}{(1-r)^n}\right),~r\uparrow 1
\end{equation*}
holds if and only if
\begin{equation*}
\left(\int_S\lambda^p\left(C(\xi,\delta)\right)d\sigma(\xi)\right)^{\frac{1}{p}}=O\left(\Phi(\delta)\right),~0<\delta<1.
\end{equation*}


{Faculty of Mechanics and Mathematics,
Ivan Franko National University of Lviv,
 Universytets'ka 1,
79000, Lviv}

{chyzhykov@yahoo.com}, {urkevych@gmail.com} 
\end{document}